\documentclass[11pt]{amsart}

\usepackage{cancel}
\usepackage{latexsym, amssymb, amsmath,esint}
\usepackage{soul}
\usepackage{amsfonts, graphicx}
\usepackage{graphicx,color}
\newcommand{\be}{\begin{equation}}
\newcommand{\ee}{\end{equation}}
\newcommand{\beq}{\begin{eqnarray}}
\newcommand{\eeq}{\end{eqnarray}}

\usepackage{hyperref}

\usepackage{wasysym,stmaryrd}
\newtheorem{thm}{Theorem}[section]

\newtheorem{lma}{Lemma}[section]
\newtheorem{prop}{Proposition}[section]
\newtheorem{cor}{Corollary}[section]
\newtheorem{defn}{Definition}[section]
\theoremstyle{remark}
\newtheorem{rem}{Remark}[section]
\numberwithin{equation}{section}

\def\be{\begin{equation}}
\def\ee{\end{equation}}
\def\bee{\begin{equation*}}
\def\eee{\end{equation*}}

\def\p{\partial}

\newcommand{\pr}{\partial}

\begin{document}

\title{Sharp interior gradient estimate for area decreasing graphical mean curvature flow in arbitrary codimension} %
\date{\today}

\author{Jingbo Wan}
\address[Jingbo Wan]{Department of Mathematics, Columbia University, New York, NY, 10027}
 \email{jingbowan@math.columbia.edu}

\renewcommand{\subjclassname}{
  \textup{2020} Mathematics Subject Classification}

\date{\today}

\begin{abstract}
 We prove the sharp interior gradient estimate for area decreasing graphical mean curvature flow in arbitrary codimension, which generalizes the result in \cite{CM}.
 \end{abstract}

\maketitle

\markboth{Jingbo Wan}{Interior gradient estimate for area decreasing graphical mean curvature flow}

\section{Introduction}

A one-parameter family of smooth submanifold $\{M_t^n\}\subset \mathbb R^{n+m}$ flows by graphical mean curvature flow if
\begin{equation}
    \pr_t \textbf{z}=\textbf{H}(\textbf{z})=\Delta_{M_t} \textbf{z},
\end{equation}

where $\textbf{z}$ are coordinates on $\mathbb R^{n+m}$ s.t. 
each $M_t$ is assumed to be the graph of a map $u(\cdot,t)=\Big(u^{n+1}(\cdot,t),...,u^{n+m}(\cdot,t)\Big)$. So if $\textbf{z}=(\textbf{x},\textbf{y})$ with $\textbf{x}\in B_1^n$, then $M_t$ is given by
\[y^\alpha=u^\alpha (\textbf{x},t),\quad \alpha=n+1,...,n+m.\]

We will always use the induced gradient $\nabla^{M_t}$ and induced Laplacian $\Delta_{M_t}$ on $M_t$.

In this paper, we prove sharp gradient estimate for higher codimensional graphical mean curvature flow that satisfies the area decreasing condition.

\begin{defn}
A map $u:B_1^n\subset\mathbb R^n\rightarrow \mathbb R^m$ is said to be area decreasing if the Jacobian of $du:\mathbb R^n\rightarrow \mathbb R^m$ on any two dimensional subspace of $\mathbb R^n$ is less than one.
\end{defn}

 Proving a priori estimate is an essential step for both minimal surface system and graphical mean curvature flow. In the study of minimal surface systems, (i.e., the graph $\Sigma^n$ of $u$ is minimal in $\mathbb{R}^{n+m}$), several important contributions have been made. For the case when $n=2$ and $m=1$, Finn \cite{F1} established that the gradient bound is given by $\log|du|(0) \leq C(1+r^{-1}||u||_\infty)$. In the general hypersurface case with $m=1$, Bombieri, De Giorgi, and Miranda \cite{BDM} proved a similar bound, $\log|du|(0) \leq C(1+r^{-1}||u||_\infty)$, and Finn's example \cite{F2} demonstrated that the linear dependence in $||u||_\infty$ on right hand side of this estimate is sharp. Korevaar \cite{K} provided a simpler proof using the Maximum Principle, yielding a weaker estimate where the right hand side has a quadratic dependence on the $C^0$ norm: $\log|du|(0) \leq C(1+r^{-1}||u||_\infty^2)$. We refer readers to Chapter 16 of \cite{GiTr} for a complete expository of this topic. For higher codimensional cases ($k>1$), Mu-Tao Wang \cite{W2} obtained a similar interior gradient estimate under an area non-increasing condition. See the paper \cite{W2} or Appendix \ref{app:MSS} for more details.

In the context of graphical mean curvature flow $\Sigma_t^n \subset \mathbb{R}^{n+m}$, Ecker and Huisken \cite{EH} adapted Korevaar’s argument to derive an estimate for the gradient bound in the hypersurface case ($m=1$): $\log |du|(0,\frac{r^2}{4n}) \leq \frac{1}{2}\log(1+||du(\cdot,0)||_\infty^2 )+C (1+r^{-1}||u(\cdot,0)||_\infty )^2$, in which the right hand side also depends on the initial gradient bound. Colding and Minicozzi \cite{CM} improved this by showing a sharp interior gradient estimate for graphical mean curvature flow: $\log |du|(0,\frac{r^2}{4n}) \leq C (1+r^{-1}||u(\cdot,0)||_\infty )^2$. They made use of the grim reaper as barriers, to demonstrate the sharpness of their estimate.

In this paper, we are able to generalize Colding-Minicozzi's sharp interior gradient estimate in \cite{CM} to higher codimensional case when the initial graph is area decreasing. Our main theorem can be stated as the following:

\begin{thm}\label{thm:main}  There exists $K_1 = K_1(n,m), K_2=K_2(n,m)$ so if the graph of $u =\Big(u^{n+1} ,...,u^{n+m} \Big): B_1^n\times [0,1]\rightarrow \mathbb R^m$ flows by mean curvature with initial map $u(\cdot,0)$ being area decreasing, then
\begin{equation}\label{eqn:int-grad-est}
        |du|(\textbf{0},\frac{1}{4n})\leq K_1 e^{K_2 ||u(\cdot,0)||^2_\infty}.
\end{equation}

 Here $||u(\cdot,0)||_\infty=\max_\beta||u^\beta(\cdot,0)||_\infty$ and $|du|=\sqrt{\mathrm{tr}((du)^T du)}$.
\end{thm}

\begin{rem}
   We assume the initial map is area decreasing in the main theorem because proving the localized version of the preservation of the area-decreasing condition along graphical mean curvature flow is straightforward (see Section \ref{sec:preserve-area-decreasing}). We expect the same result to hold if the initial map is only area non-increasing, i.e., the Jacobian of the differential on any two planes is less than or equal to one. The challenge will be localizing the strong maximum principle argument in \cite{LTW}.
\end{rem}

\begin{rem}
    The graphical mean curvature flow in the hypersurface case has a scalar height function whose differential is of rank one, ensuring the area-decreasing condition is always satisfied. Consequently, our main theorem recovers the result proved in Theorem 1 of \cite{CM} in codimension one case. The quadratic dependence of the exponent in \eqref{eqn:int-grad-est} on $||u(\cdot,0)||^2_\infty$ is sharp, as demonstrated by a specific grim reaper solution constructed in Proposition 1 of \cite{CM}.
\end{rem}

The organization of this paper is as follows: In Section \ref{sec:preserve-area-decreasing}, we follow the localization idea of \cite{EH} to establish a localized version of the preservation of area-decreasing property along the graphical mean curvature flow. In Section \ref{sec:evo-eqn}, we derive the evolution equation for the volume element along the area-decreasing graphical mean curvature flow, which serves as a crucial component for the Maximum Principle argument in Section \ref{sec:proof-main}. In Section \ref{sec:proof-main}, we employ a Korevaar-type maximum principle argument as illustrated in \cite{CM}, extending the computations in \cite{CM} to the higher codimensional case. We note that the area-decreasing condition is fundamental to this paper. Finally, in Appendix \ref{app:MSS}, we present a self-contained, refined version of the maximum principle proof for the interior gradient estimate for the minimal surface system in \cite{W2}.

\vskip0.2cm

{\it Acknowledgement}: The author would like to thank Prof. Mu-Tao Wang for his his encouragement, valuable discussions and also for introducing this problem.

\section{Localized preservation of area decreasing condition}\label{sec:preserve-area-decreasing}

Given any map $u: (\mathbb R^n,g_{\mathbb R^n})\rightarrow (\mathbb R^m,g_{\mathbb R^m})$, following the idea of \cite{TW}, we consider the parallel two tensor $s=\pi_{\mathbb R^n}^* g_{\mathbb R^n} - \pi_{\mathbb R^m}^* g_{\mathbb R^m}$ in the ambient product space $\mathbb R^n\times \mathbb R^m$: for any $X,Y\in \Gamma (T (\mathbb R^{n}\times \mathbb R^m))$,
   \[s(X,Y):= g_{\mathbb R^m}(\pi_{\mathbb R^n} (X),\pi_{\mathbb R^n}(Y))- g_{\mathbb R^n}(\pi_{\mathbb R^m}(X),\pi_{\mathbb R^m}(Y)),  \]
   where $\pi_{\mathbb R^n}:\mathbb R^n\times \mathbb R^m\rightarrow \mathbb R^n$ and $\pi_{\mathbb R^m}:\mathbb R^n\times \mathbb R^m \rightarrow \mathbb R^m$ are the natural projections.

Suppose that the graph $\{M_t\}$ of $u(\cdot,t)$ is a mean curvature flow. Let $F_t$ be the embedding $x \in U\subset \mathbb R^n \rightarrow (x, u(x,t)) \in U\times \mathbb R^m$. Consider the restriction $S$ of $s$ on $F_t(U)=M_t$, i.e. $S=F_t^*s$. Given any space-time point $p$, we can do a singular value decomposition of the differential $du$ with singular value $\lambda_1\geq \lambda_2\geq ...\geq \lambda_n\geq 0.$
 
Choosing an appropriate orthonormal basis around $p$, we can diagonalize $S|_p$ so that $S_{ij}|_p=S_{ii}\delta_{ij}$. Then eigenvalues of $S$ are given by 
\[S_{ii}=\frac{1-\lambda_i^2}{1+\lambda_i^2},\quad i=1,...,n, \]

so that $S$ being positive at $p$ is equivalent to $u$ being distance decreasing at $p$.

In section 5 of \cite{TW}, one introduced a tensor $S^{[2]}$ from $S$, which can be viewed as a symmetric endomorphism on $\Lambda^2(T^*M_t)$.  With respect to an orthonormal frame,
\begin{align}
S^{[2]}_{(ij)(k\ell)} &= S_{ik}\delta_{j\ell} + S_{j\ell}\delta_{ik} - S_{i\ell}\delta_{jk} - S_{jk}\delta_{i\ell} \label{def_S2}
\end{align}
for any $i<j$ and $k<\ell$. Notice that $S^{[2]}$ being positive is equivalent to two positivity of $S$. Under the diagonalization of $S|_p$, we can write down eigenvalues of $S^{[2]}$ in terms of singular values of $du_p$
\[S^{[2]}_{(ij)(ij)}=S_{ii}+S_{jj}=\frac{2(1-\lambda_i^2\lambda_j^2)}{(1+\lambda_i^2)(1+\lambda_j^2)},\quad 1\leq i< j\leq n,   \]

so that $S^{[2]}$ being positive at $p$ is equivalent to $u$ being area decreasing at $p$.

We need the following calculation lemma for the evolution of $S_{ii}+S_{jj}$, $i< j$ from \cite{TTW}. See (3.3) in \cite{TTW} and notice we don't have any curvature terms here because the domain and target are flat.

\begin{lma}[cf. Lemma 3.1 in \cite{TTW}] \label{lem_Sij}
Suppose that $S^{[2]}$ is positive definite.  For any $i,j$ with $1\leq i<j\leq n$,
\begin{align} \begin{split}
 (S_{ii}+S_{jj})^{-1}(\pr_t-\Delta_{M_t})(S_{ii}+S_{jj}) + \frac{1}{2} (S_{ii}+S_{jj})^{-2}|\nabla^{M_t}(S_{ii}+S_{jj})|^2 \geq 0.
\end{split} 
\end{align}
\end{lma}

\begin{lma}[cf. Theorem 3.2 in \cite{TTW}]
    If $S^{[2]}$ is positive definite, then
    \begin{equation}\label{eqn:heat-logdetS2}
        \Big(\p_t-\Delta_{M_t}\Big) \log \det S^{[2]}\geq \frac{1}{2}|\nabla^{M_t}  \log \det S^{[2]} |^2.
    \end{equation}

    As a result, the scalar function $\Phi(\textbf{z},t)$ defined as
    \[\Phi:= 1+\frac{n(n-1)}{2}\log 2- \log \det S^{[2]}  \]

    satisfies 
    \begin{equation}\label{eqn:heat-Phi}
      \Big(\p_t-\Delta_{M_t}\Big)\Phi\leq -\frac{|\nabla^{M_t}\Phi |^2}{2\Phi}.
    \end{equation}
\end{lma}
\begin{proof}
We follow closely the proof of Theorem 3.2 in \cite{TTW}. Suppose that $S^{[2]}$ is positive definite at $p$, hence in a space-time neighborhood of $p$, and denote its inverse by $Q$ .  Note that $Q^{(ij)(k\ell)}|_p = (S_{ii} + S_{jj})^{-1}\delta^{ik}\delta^{j\ell}$. At a given spacetime point $p$, we compute
\begin{align}
(\pr_t-\Delta_{M_t})\log\det S^{[2]}
=& Q^{AB}\left[(\pr_t-\Delta_{M_t})S^{[2]}_{AB}\right] - (\nabla_\ell Q^{AB})(\nabla_\ell S^{[2]}_{AB}) \notag \\
=& Q^{AB}\left[(\pr_t-\Delta_{M_t})S^{[2]}_{AB}\right] + Q^{AC}(\nabla_\ell S^{[2]}_{CD})Q^{DB}(\nabla_\ell S^{[2]}_{AB}) \notag \\
{}^{\text{at}\ p}=& \sum_A Q^{AA}\left[(\pr_t-\Delta_{M_t})S^{[2]}_{AA}\right] + \sum_{A,B} Q^{AA}Q^{BB} \left|\nabla^{M_t} S^{[2]}_{AB}\right|^2 ~. \label{lndet01}
\end{align}
The first term on the last line of \eqref{lndet01} is
\begin{align}
\sum_{1\leq i<j\leq n}(S_{ii}+S_{jj})^{-1}(\pr_t-\Delta_{M_t})(S_{ii}+S_{jj}) ~. \label{lndet02}
\end{align}
For the second term, it follows from the definition \eqref{def_S2} that $S^{[2]}_{(ij)(k\ell)} \equiv 0$ when $i,j,k,\ell$ are all distinct.  If $i,j,k$ are all distinct, $S^{[2]}_{AB} = \pm S_{k\ell}$ for $A = (ij)$ or $(ji)$, and $B = (ik)$ or $(ki)$.  With this understood, the second term of \eqref{lndet01} becomes
\begin{align}
&\sum_{1\leq i<j\leq n} (S_{ii}+S_{jj})^{-2}|\nabla^{M_t}(S_{ii}+S_{jj})|^2 \nonumber\\
&+ 2\sum_{1\leq i\leq n}\sum_{\substack{1\leq j<k\leq n\\j\neq i,\,k\neq i}}(S_{ii}+S_{jj})^{-1}(S_{ii}+S_{kk})^{-1}|\nabla^{M_t}S_{jk}|^2 ~. \label{lndet03}
\end{align}

Now combine \eqref{lndet01}-\eqref{lndet03} with Lemma \ref{lem_Sij}, we see
\begin{align*}
    (\pr_t-\Delta_{M_t})\log\det S^{[2]}\geq &\sum_{1\leq i<j\leq n}(S_{ii}+S_{jj})^{-1}(\pr_t-\Delta_{M_t})(S_{ii}+S_{jj})\\
    &+\sum_{1\leq i<j\leq n} (S_{ii}+S_{jj})^{-2}|\nabla^{M_t}(S_{ii}+S_{jj})|^2 \\
    \geq& \frac{1}{2}\sum_{1\leq i<j\leq n} (S_{ii}+S_{jj})^{-2}|\nabla^{M_t}(S_{ii}+S_{jj})|^2\\
    =& \frac{1}{2}|\nabla^{M_t}  \log \det S^{[2]} |^2,
\end{align*}
and this proves \eqref{eqn:heat-logdetS2}.

Now for \eqref{eqn:heat-Phi}, we first express $\Phi$ using singular values of $du$:
    \[\Phi= 1+\log \prod_{1\leq i<j\leq n} \frac{(1+\lambda_i^2)(1+\lambda_j^2)}{1-\lambda_i^2\lambda_j^2},   \]

    so that $S^{[2]}$ being positive definite implies $\Phi\geq 1$. Therefore, we get the desired equation
    \begin{align*}
    \Big(\p_t-\Delta_{M_t}\Big)\Phi\leq -\frac{1}{2}|\nabla^{M_t}\Phi |^2\leq -\frac{|\nabla^{M_t}\Phi |^2}{2\Phi}.    
    \end{align*}
\end{proof}

Next, we provide a lemma that an upper bound of $\Phi$ will imply the area decreasing condition.

\begin{lma}\label{lem:Phi-algebra}
    If $\Phi\leq C_0$, then
    \[\lambda_i^2\lambda_j^2\leq 1-e^{1-C_0},\quad 1\leq i< j\leq n  \]
\end{lma}

\begin{proof} Suppose we have ordered the singular values as $\lambda_1\geq \lambda_2\geq ...\geq \lambda_n\geq 0$, then
    \begin{align*}
&\Phi= 1+\log \prod_{1\leq i<j\leq n} \frac{(1+\lambda_i^2)(1+\lambda_j^2)}{1-\lambda_i^2\lambda_j^2}\leq C_0\\      \Longrightarrow &  \log \prod_{1\leq i<j\leq n} \frac{1-\lambda_i^2\lambda_j^2} {(1+\lambda_i^2)(1+\lambda_j^2)}\geq 1- C_0\\
\Longrightarrow & \log (1-\lambda_1^2\lambda_2^2)\geq 1-C_0\\
\Longrightarrow & \lambda_1^2\lambda_2^2\leq 1-e^{1-C_0}.
    \end{align*}
\end{proof}

The idea of localization we use here is the same as the one in \cite{EH}, in which the following cut-off function is the key ingredient.

\begin{lma}\label{lem:cutoff-for-area-decreasing} Suppose the graph $M_t$ of $u$ flows by mean curvature and $\textbf{z}$ is its position vector. Fix any $R>0$, let $\varphi$ be a function defined by
\begin{equation}
     \varphi=\left(R^2-|\textbf{z}|^2-2nt \right)^2,
\end{equation}
then in the support of $\left(R^2-|\textbf{z}|^2-2nt \right)_+$, it satisfies
\begin{equation}\label{eqn:varphi-eqn}
    \Big(\pr_t-\Delta_{M_t}\Big)\varphi=-\frac{1}{2}\varphi^{-1}|\nabla^{M_t} \varphi|^2.
\end{equation}
\end{lma}
\begin{proof}
    We can directly compute    
    \begin{align*}
     \pr_t\varphi=&   2\left(R^2-|\textbf{z}|^2-2nt \right)\left(- \pr_t |\textbf{z}|^2 -2n \right),\\
     \nabla^{M_t}\varphi=& 2\left(R^2-|\textbf{z}|^2-2nt \right)\left(-\nabla^{M_t}|\textbf{z}|^2 \right),\\
     \Delta_{M_t}\varphi=&2\left(R^2-|\textbf{z}|^2-2nt \right)\left(-\Delta_{M_t}|\textbf{z}|^2 \right)+2\left|\nabla^{M_t}|\textbf{z}|^2\right|^2.
    \end{align*}

    Therefore, in the support of $\left(R^2-|\textbf{z}|^2-2nt \right)_+$, we have
    \begin{align*}
    \Big(\pr_t-\Delta_{M_t}\Big)\varphi=&    2\left(R^2-|\textbf{z}|^2-2nt \right)\left(- (\pr_t-\Delta_{M_t}) |\textbf{z}|^2 -2n \right)-2\left|\nabla^{M_t}|\textbf{z}|^2\right|^2\\
    =&-\frac{1}{2\varphi}\times
    4\left(R^2-|\textbf{z}|^2-2nt \right)^2\left|\nabla^{M_t}|\textbf{z}|^2 \right|^2\\
    =&-\frac{1}{2}\varphi^{-1}|\nabla^{M_t} \varphi|^2,
    \end{align*}
    where in the second equality, we have used the fact that $(\pr_t-\Delta_{M_t}) |\textbf{z}|^2=-2n$ for mean curvature flow.
\end{proof}

\begin{prop}\label{prop:local-area-decreasing-preservation}
    Along graphical mean curvature flow $M_t$, we have the estimate $\varphi(\textbf{z},t)^4\Phi(\textbf{z},t)\leq\sup_{M_0}\varphi^4 \Phi $, therefore,
    \begin{equation}\label{eqn:local-area-decreasing-preservation}
       \left(R^2-|\textbf{z}|^2-2nt \right)_+ \Phi (\textbf{z},t)^{1/8} \leq \sup_{M_0} \left(R^2-|\textbf{z}|^2 \right)_+ \Phi (\textbf{z},0)^{1/8},
    \end{equation}

    as long as $\Phi(\textbf{z},t)$ is defined everywhere on the support of $\left(R^2-|\textbf{z}|^2-2nt \right)_+$.
\end{prop}

\begin{proof}
If $\Phi(\textbf{z},t)$ is defined everywhere on the support of $\left(R^2-|\textbf{z}|^2-2nt \right)_+$, we can make use of \eqref{eqn:heat-Phi}, \eqref{eqn:varphi-eqn} to compute the evolution of $\varphi \Phi^{1/4}$:
    \begin{align*}
\Big(\p_t-\Delta_{M_t}\Big) (\varphi \Phi^{1/4})=&\Phi^{1/4} \Big(\p_t-\Delta_{M_t}\Big) \varphi        + \frac{1}{4}\varphi\Phi^{-3/4}\Big(\p_t-\Delta_{M_t}\Big) \Phi -2  \nabla^{M_t} \varphi\cdot \nabla^{M_t} \Phi^{1/4}\\
\leq& -\frac{1}{2} \Phi^{1/4} \varphi^{-1}\left| \nabla^{M_t}\varphi\right|^2-\frac{1}{8}\varphi \Phi^{-7/4}\left| \nabla^{M_t}\Phi\right|^2-\frac{1}{2}\Phi^{-3/4}\nabla^{M_t} \varphi\cdot \nabla^{M_t} \Phi\\
=&-\frac{1}{2}\varphi\Phi^{1/4}\left|\frac{\nabla^{M_t}\varphi}{\varphi} +\frac{1}{2} \frac{\nabla^{M_t}\Phi}{\Phi}    \right|^2\\
\leq & 0.
    \end{align*}
The weak parabolic maximum principle then implies the result.
\end{proof}

\begin{cor}\label{cor:area-decreasing}
If the graph $M_t$ of $u(\cdot,t):B_1^n\rightarrow \mathbb R^m$ flows by mean curvature, and the initial map $u(\cdot,0)$ is area decreasing in $B_1^n$, then $u\left(\cdot, \frac{1}{4n}\right)$ is area decreasing at the origin $\textbf{0}\in B_1^n$.
\end{cor}
\begin{proof}
    We choose $R^2=1+ \sum_{\alpha=n+1}^{n+m} ||u^\alpha (\cdot,0)||_\infty^2 $. Since the initial map $u(\cdot,0)$ is area decreasing, \eqref{eqn:local-area-decreasing-preservation} implies
    \[\left(R^2-|\textbf{z}|^2-2nt \right)_+ \Phi (\textbf{z},t)^{1/8} \leq C_0<\infty,\]
    for some constant $C_0$ depending on $M_0$. By the maximum principle, we know
    \[\sum_{\alpha=n+1}^{n+m} ||u^\alpha (\textbf{0},\frac{1}{4n} )||_\infty^2\leq \sum_{\alpha=n+1}^{n+m} ||u^\alpha (\cdot,0)||_\infty^2. \] So at $t=\frac{1}{4n}$ and $\textbf{x}=\textbf{0}$, $\textbf{y}=u\left(\textbf{0}, \frac{1}{4n}\right)$,
    \[\Phi\left( \left(\textbf{0}, u\left(\textbf{0}, \frac{1}{4n}\right)\right) ,\frac{1}{4n}\right)\leq  (2C_0)^8<\infty,  \]
    then the area decreasing property of $u\left(\cdot, \frac{1}{4n}\right)$ at $\textbf{0}$ follows from Lemma \ref{lem:Phi-algebra}.
\end{proof}

\section{A crucial evolution equation along area decreasing GMCF}\label{sec:evo-eqn}

In this section, we include the derivation of an essential lemma, which states that the volume element $ v$ of the graphical mean curvature flow $ u$ is, in a certain sense, a supersolution to a heat equation, given that the area decreasing condition is satisfied. This crucial lemma indicates the essential nature of area decreasing condition for interior gradient estimate of higher codimensional graphical mean curvature flow. This evolution equation was first derived in \cite{W1}, and the elliptic version can also be seen in Lemma 2.2 of \cite{W2}.

\begin{lma}\label{lem:w-super-heat}
   Suppose the graph of $ u:B^n_{1}\times \left[0, 1\right] \longrightarrow \mathbb R^m $ flows by mean curvature and satisfies the area decreasing condition, then $\frac{1}{n}$-power of the volume element $ v=\sqrt{\det\left(I+(du)^T du  \right)}$, $ w= v^{\frac{1}{n}}$ satisfies
  \begin{equation}\label{eqn:w-super-heat}
   (\pr_t-\Delta_{M_t})  w\leq -\frac{2|\nabla^{M_t}  w|^2}{ w}.   
  \end{equation}
\end{lma}

\begin{proof}
    For the convenience of the readers, we include a derivation of this lemma starting from Equation (3.8) of \cite{W1}.  Since our domain and target are flat, we have from Equation (3.8) of \cite{W1} that the reciprocal of the volume element $v^{-1}=*\Omega_1$ evolves along graphical mean curvature flow as
    \begin{equation}\label{eqn:heat-v1}
     \begin{split}
          &(\pr_t-\Delta_{M_t})   v^{-1}\\
          &=   v^{-1} \cdot \left(\sum_{\alpha}\sum_{i,k} h^2_{\alpha,ik} - 2\sum_{k; i<j} \lambda_i\lambda_j h_{n+i,ik}h_{n+j,jk} +2 \sum_{k; i<j} \lambda_i\lambda_j h_{n+j,ik} h_{n+i,jk}   \right),
     \end{split}  
    \end{equation}
    where $\lambda_1\geq \lambda_2\geq ...\geq \lambda_n\geq 0$ are the singular values of $d  u$ at a given point, and $h_{\alpha,ik}$ are components of the second fundamental form. Now notice that 
    \begin{align}\label{eqn:heat-v2}
    |\nabla^{M_t}\log   v |^2=& \sum_{k=1}^n \left(\sum_{i=1}^n \lambda_i h_{n+i,ik} \right)^2 \\
    =&2\sum_{k=1}^n\sum_{1\leq i<j\leq n} \lambda_i\lambda_j h_{n+i,ik}h_{n+j,jk}+\sum_{k,i=1}^n\lambda_i^2 h_{n+i,ik}^2, \nonumber    
    \end{align}
    see the proof of Lemma 2.2 in \cite{W2}. Combining \eqref{eqn:heat-v1} and \eqref{eqn:heat-v2}, we get
    \begin{align}\label{eqn:heat-v3}
     &(\pr_t-\Delta_{M_t})  \log   v = -\frac{(\pr_t-\Delta_{M_t})   v^{-1}}{  v^{-1}} - | \nabla^{M_t}\log   v|^2\\
     =&-\sum_{\alpha=n+1}^{n+m}\sum_{i,k=1}^n h^2_{\alpha,ik}-\sum_{k,i=1}^n\lambda_i^2 h_{n+i,ik}^2-2 \sum_{k=1}^n\sum_{1\leq i<j\leq n} \lambda_i\lambda_j h_{n+j,ik} h_{n+i,jk} . \nonumber
    \end{align}

    Equation (2.2) of \cite{W2} is the elliptic version of this. Notice that $\lambda_i=0$ for $i>\ell:=\min(n,m)$, so the first term and the third term in \eqref{eqn:heat-v3} can complete a square:
    \begin{align*}
    &-\sum_{\alpha=n+1}^{n+m}\sum_{i,k=1}^n h^2_{\alpha,ik}  -2 \sum_{k=1}^n\sum_{1\leq i<j\leq n} \lambda_i\lambda_j h_{n+j,ik} h_{n+i,jk} \\
    \leq& -\sum_{k=1}^n\left(\sum_{i,j=1}^\ell h_{n+j,ik}^2 +2 \sum_{1\leq i<j\leq \ell} \lambda_i\lambda_j h_{n+j,ik} h_{n+i,jk}\right)\\
    \leq &-\sum_{k=1}^n\sum_{1\leq i<j\leq \ell}\left(h_{n+j,ik}^2 - |h_{n+j,ik}||h_{n+i,jk}|+ h_{n+i,jk}^2  \right)\\
    =&-\sum_{k=1}^n\sum_{1\leq i<j\leq \ell}\Big(  |h_{n+j,ik}|-|h_{n+i,jk}|\Big)^2\leq 0,
    \end{align*}
    where we have used the area decreasing condition $\lambda_i\lambda_j< 1 $ for $i\neq j$ in the second inequality. Therefore, from \eqref{eqn:heat-v3}, we see
    \begin{align}
     (\pr_t-\Delta_{M_t})  \log   v\leq    -\sum_{k,i=1}^n\lambda_i^2 h_{n+i,ik}^2 \leq -\frac{1}{n}\sum_{k=1}^n \left(\sum_{i=1}^n \lambda_i h_{n+i,ik} \right)^2=-\frac{1}{n}|\nabla^{M_t}\log   v |^2.
    \end{align}

    This implies $  w=  v^{\frac{1}{n}}$ satisfies
    \[ (\pr_t-\Delta_{M_t})  \log   w\leq -|\nabla^{M_t}\log   w |^2 , \]
    hence
    \[(\pr_t-\Delta_{M_t})  w=   w (\pr_t-\Delta_{M_t})  \log   w - \frac{|\nabla^{M_t}   w|^2}{  w}\leq - \frac{2|\nabla^{M_t}   w|^2}{  w}  \]
    Therefore, we conclude the proof of the desired evolution equation \eqref{eqn:w-super-heat} along graphical mean curvature flow, assuming the area decreasing condition.
    
\end{proof}

We remark that \eqref{eqn:w-super-heat} will play a similar role as equation (1.6) in \cite{CM}, i.e., the evolution equation of the volume element in the hypersurface case.

\section{Proof of the main theorem}\label{sec:proof-main}

In this section, we will conclude the proof of the main theorem. The main strategy involves constructing a cut-off function similar to the one in Lemma 2 of \cite{CM} and applying the maximum principle as per the method in \cite{K} to derive the gradient upper bound. Before proceeding with these steps, we will perform a parabolic rescaling to simplify the situation .

\subsection{A simplification through parabolic rescaling}\label{subsec:parabolic-rescaling}

Now suppose the graph $\{M_t\}$ of $u(\cdot ,t): B^n_1\times [0,1]\rightarrow \mathbb R^m$ flows by mean curvature in $\mathbb R^{n+m}$ where $\textbf{z}$ is its position vector at given time $t$. We write $\Lambda=||u(\cdot,0)||_\infty\geq 0$, and we can do a parabolic rescaling with translation along $\textbf{y}$-directions as follows:
    \[(\textbf{z},t) \mapsto (\bar{\textbf{z}},\bar t )=\left(\frac{\textbf{z}}{1+2\Lambda} +(\textbf{0},\textbf{y}_0)  ,\frac{t}{(1+2\Lambda)^2}  \right),   \]

where $(\textbf{0},\textbf{y}_0)=\left(0,...,0,\frac{1+3\Lambda}{1+2\Lambda},...,\frac{1+3\Lambda}{1+2\Lambda}     \right)$ with all $x^i$ components being zero and all $y^\alpha$ components being $\frac{1+3\Lambda}{1+2\Lambda}$ represents a translation in the $\textbf{y}$-directions.

Therefore, we can write $\textbf{z}=(\textbf{x},u^\alpha(\textbf{x},t))$ and $\bar{\textbf{z}}=( \bar {\textbf{x}}, \bar u^\alpha(\bar {\textbf{x}},\bar t))$. So
    \[\bar {\textbf{x}}=\frac{ \textbf{x}}{1+2\Lambda},\quad \bar u^\alpha(\bar {\textbf{x}},\bar t) =\frac{u^\alpha(\textbf{x},t)}{1+2\Lambda} +\frac{1+3\Lambda}{1+2\Lambda}   ,\quad \bar t =\frac{t}{(1+2\Lambda)^2}, \]
and hence
\begin{equation}\label{eqn:height-function-transform}
    \bar u^\alpha (\bar {\textbf{x}},\bar t)=\frac{u^\alpha \Big( (1+2\Lambda) \bar {\textbf{x}}, (1+2\Lambda)^2\bar t \Big)  }{1+2\Lambda} +\frac{1+3\Lambda}{1+2\Lambda} .
\end{equation}
As a result of this transformation, $(\bar {\textbf{z}},\bar t)$ is a mean curvature flow of the graph of function $\bar u (\cdot,\bar t): B^n_{\frac{1}{1+2\Lambda}}\times \left[0, \frac{1}{(1+2\Lambda)^2}\right] \longrightarrow \mathbb R^m $ with the initial heights
\[\bar u^\alpha(\bar {\textbf{x}},0)= \frac{u^\alpha \Big( (1+2\Lambda) \bar {\textbf{x}}, 0 \Big)  }{1+2\Lambda} +\frac{1+3\Lambda}{1+2\Lambda}  \in [1,2],\quad \alpha=n+1,...n+m. \]

By the maximum principle, we know $1 \leq \bar u^\alpha \leq 2$ for all time. Moreover, the singular values of $d\bar u$ at a given point $(\bar {\textbf{x}},\bar t)$ are the same as singular values of $d u$ at the corresponding point $(\textbf{x},t)=\Big((1+2\Lambda)\bar{\textbf{x}},(1+2\Lambda)^2\bar t \Big)$. So if $u$ is an area decreasing map, so is $\bar u$. The norms of their gradients are related as follows
\[|d\bar u|(\bar {\textbf{x}},\bar t) = |du|\Big((1+2\Lambda) \bar {\textbf{x}}, (1+2\Lambda)^2\bar t \Big)=|du|(\textbf{x},t).     \]

We will slightly abuse notations by dropping the upper-bar and continue to use $\{M_t\}$, $u(\textbf{x},t): B_{\frac{1}{1+2\Lambda}}\times \left[0, \frac{1}{(1+2\Lambda)^2}\right]\rightarrow [1,2]^m$, and $\textbf{z}$ to denote the transformed graphical mean curvature flow, its height functions, and its position vectors.

\subsection{Computation for Colding-Minicozzi's cut-off function in higher codimension}

Now we consider a direct generalization of the cut-off function introduced in \cite{CM} to higher codimensional graphical mean curvature flow.

\begin{lma} \label{lem:CM-cut-off}

Set $\phi=\eta e^{-a|\textbf{y}|^2/t}$ for $a\geq 1$ and $\eta=\left(\frac{1}{1+2\Lambda}  -|\textbf{x}|^2-2nt\right)_+$. If the graph of $ u:B^n_{\frac{1}{1+2\Lambda}}\times \left[0, \frac{1}{(1+2\Lambda)^2}\right] \longrightarrow [1,2]^m $ flows by mean curvature, then 

\begin{equation}\label{eqn:cut-off-phi}
    (\pr_t-\Delta_{M_t})\phi\leq  \frac{e^{-a|  u|^2/t}}{t^2}\left(a\eta \sum_{\alpha=n+1}^{n+m} (  u^\alpha)^2 -  \sum_{i=1}^n \frac{2a^2\eta (  u^{n+i})^2\lambda_i^2}{1+\lambda_i^2}  +\sum_{i=1}^n \frac{ 8a \lambda_i   u^{n+i}}{1+\lambda_i^2} \right),
\end{equation}

provided $\eta>0$ and $t>0$. 
Here $\lambda_1\geq \lambda_2\geq ....\geq \lambda_n\geq 0$ are singular values of $d u$ at the given point.
\end{lma}

\begin{proof} 
   Suppose $\eta>0$ and $t>0$. For mean curvature flow, it's well-known that $(\pr_t-\Delta_{M_t})\eta \leq 0$, see (1.6) \cite{CM}. Now we directly differentiate to compute:
    \begin{align*}
\pr_t \left(\eta e^{-a|\textbf{y}|^2/t}\right)=&\pr_t \eta\cdot   e^{-a|\textbf{y}|^2/t} + \eta\left(-\frac{a}{t}\pr_t |\textbf{y}|^2 + \frac{a|\textbf{y}|^2}{t^2}  \right)     e^{-a|\textbf{y}|^2/t} , \\
\nabla^{M_t}\left(\eta e^{-a|\textbf{y}|^2/t}\right)=& \nabla^{M_t} \eta\cdot  e^{-a|\textbf{y}|^2/t} -\eta \frac{a}{t} \nabla^{M_t} |\textbf{y}|^2\cdot  e^{-a|\textbf{y}|^2/t}, \\
\Delta_{M_t}\left(\eta e^{-a|\textbf{y}|^2/t}\right)=& \Delta_{M_t} \eta\cdot  e^{-a|\textbf{y}|^2/t} -  \frac{2a}{t} \langle\nabla^{M_t} \eta, \nabla^{M_t} |\textbf{y}|^2\rangle \cdot  e^{-a|\textbf{y}|^2/t}\\
&-\eta \frac{a}{t} \Delta_{M_t} |\textbf{y}|^2\cdot  e^{-a|\textbf{y}|^2/t}+\eta \frac{a^2}{t^2} \left|\nabla^{M_t} |\textbf{y}|^2\right|^2\cdot  e^{-a|\textbf{y}|^2/t}.
    \end{align*}

Therefore, we can combine and get
\begin{align}\label{eqn:heat-phi1}
    &(\pr_t-\Delta_{M_t}) \phi \\
    =&(\pr_t-\Delta_{M_t}) \eta  \cdot e^{-a|\textbf{y}|^2/t} + \eta\left(-\frac{a}{t}(\pr_t-\Delta_{M_t})  |\textbf{y}|^2 + \frac{a|\textbf{y}|^2}{t^2}  \right)     e^{-a|\textbf{y}|^2/t} \nonumber\\
    &+\frac{2a}{t} \langle\nabla^{M_t}\eta , \nabla^{M_t} |\textbf{y}|^2\rangle \cdot  e^{-a|\textbf{y}|^2/t} -\eta \frac{a^2}{t^2} \left|\nabla^{M_t} |\textbf{y}|^2\right|^2\cdot  e^{-a|\textbf{y}|^2/t}  \nonumber\\
    \leq &\left(\frac{2a}{t}\eta\sum_{\alpha=n+1}^{n+m}|\nabla^{M_t}y^\alpha |^2 + \frac{a|\textbf{y}|^2}{t^2}\eta -\frac{2a}{t} \langle\nabla^{M_t}|\textbf{x} |^2 , \nabla^{M_t} |\textbf{y}|^2\rangle   -\eta \frac{a^2}{t^2} \left|\nabla^{M_t} |\textbf{y}|^2\right|^2   \right)e^{-a|\textbf{y}|^2/t},  \nonumber
\end{align}
where in the last inequality we have used $a\geq 1>0$, $\nabla^{M_t}\eta=-\nabla^{M_t}|\textbf{x}|^2$ and
\begin{align*}
    (\pr_t-\Delta_{M_t})|\textbf{y}|^2=&\sum_{\alpha=n+1}^{n+m}\Big(2y^\alpha(\pr_t-\Delta_{M_t})y^\alpha-2|\nabla^{M_t}y^\alpha |^2\Big)=-2\sum_{\alpha=n+1}^{n+m}|\nabla^{M_t}y^\alpha |^2.
\end{align*}

Since $M_t$ is graphical, the coordinate functions $\{x^i\}_{i=1}^n$ form a local coordinate system on $M_t$ and $y^\alpha=  u^\alpha(x^i,t)$, so we can rewrite \eqref{eqn:heat-phi1} as
\begin{align}\label{eqn:heat-phi2}
    &(\pr_t-\Delta_{M_t}) \phi\\
    \leq &e^{-a|  u|^2/t} \left(\frac{2a}{t}\eta \sum_{\alpha=n+1}^{n+m} \sum_{i,j=1}^n g^{ij} \pr_i   u^\alpha  \pr_j   u^\alpha  + \frac{a}{t^2}\eta \sum_{\alpha=n+1}^{n+m}(  u^\alpha )^2   \right. \nonumber\\
    &\quad\quad \left.  -\frac{2a}{t} \sum_{\alpha=n+1}^{n+m} \sum_{i,j=1}^n g^{ij}x^i   u^\alpha \pr_j   u^\alpha -\frac{4a^2}{t^2}\eta \sum_{\alpha,\beta=n+1}^{n+m} \sum_{i,j=1}^n g^{ij}   u^\alpha \pr_i   u^\alpha   u^\beta \pr_j   u^\beta\right),  \nonumber
\end{align}
where $g_{ij}$ is the induced metric on $M_t$ and $g^{ij}$ is its inverse. At a given point, we can employ the singular value decomposition of $d  u$ and choose coordinate systems $x^i$ and $y^\alpha$ appropriately such that
\[g^{ij}=\frac{1}{1+\lambda_i^2}\delta_{ij}, \quad \pr_i   u^\alpha=\lambda_i \delta_{\alpha,n+i}, \]
with $\lambda_1\geq \lambda_2\geq ...\geq \lambda_n$ being the singular values of $d  u$ at the given point. Therefore \eqref{eqn:heat-phi2} gives the desired inequality \eqref{eqn:cut-off-phi}
\begin{align*}
  & (\pr_t-\Delta_{M_t}) \phi\\
  \leq & \frac{e^{-a|  u|^2/t}}{t^2}\left(2at\eta \sum_{i=1}^n \frac{\lambda_i^2}{1+\lambda_i^2}  + a\eta \sum_{\alpha=n+1}^{n+m}(  u^\alpha )^2 -2at\sum_{i=1}^n \frac{x^i   u ^{n+i}\lambda_i }{1+\lambda_i^2} -4a^2\eta \sum_{i=1}^n \frac{(  u^{n+i})^2\lambda_i^2 }{1+\lambda_i^2}    \right)\\
   \leq &\frac{e^{-a|  u|^2/t}}{t^2} \left(a\eta \sum_{\alpha=n+1}^{n+m} (  u^\alpha)^2 -  \sum_{i=1}^n \frac{2a^2\eta (  u^{n+i})^2\lambda_i^2}{1+\lambda_i^2}  +\sum_{i=1}^n \frac{ 8a \lambda_i   u^{n+i}}{1+\lambda_i^2}\right),
\end{align*}

where in the last inequality we have used $a\geq 1, 0<t\leq 1, |x^i|\leq 1,   u^{n+i}\geq 1$ and hence
\[4a^2 (  u^{n+i})^2- 2at\geq  2a^2 (  u^{n+i})^2.\]
\end{proof}

\subsection{Korevarr type maximum principle argument}
Now we consider the maximum point $p$ of $\phi   w$ in $\bigcup_{t\in [0,\frac{1}{1+\Lambda} ] } M_t$ under the assumption that $  u$ is area decreasing map. By the choice of the cut-off $\phi$, we know $\phi   w =0$ on $M_0 \cup_t (\pr M_t)$ and $\max \phi   w>0$. Therefore, at $p$, we have
\begin{equation*}
        \nabla^{M_t}(\phi   w)=0,\quad  (\pr_t-\Delta_{M_t})(\phi   w)\geq 0.
\end{equation*}
 
The first equation implies
\[\nabla^{M_t} \phi=-\frac{\phi}{  w}  \nabla^{M_t}  w,\]

and the second equation implies
\[  w(\pr_t-\Delta_{M_t}) \phi -2\langle\nabla^{M_t} \eta ,\nabla^{M_t}  w\rangle+ \phi (\pr_t-\Delta_{M_t} )  w\geq 0. \]

From these two equations and Lemma \ref{lem:w-super-heat}, we know that if area decreasing condition holds at $p$, then at this point
\begin{equation}\label{eqn:heat-phigeq0}
    (\pr_t-\Delta_{M_t}) \phi \geq 0.
\end{equation}

By area decreasing condition, singular values of $d  u$ at this point $p$ have the ordering $\lambda_1\geq 1\geq \lambda_2\geq ....\geq \lambda_n\geq 0$. Now combine with \eqref{eqn:heat-phigeq0} and Lemma \ref{lem:CM-cut-off}, we derive at this point
\[0\leq  a\eta \sum_{\alpha=n+1}^{n+m} (  u^\alpha)^2 -  \sum_{i=1}^n \frac{2a^2\eta (  u^{n+i})^2\lambda_i^2}{1+\lambda_i^2}  +\sum_{i=1}^n \frac{ 8a \lambda_i   u^{n+i}}{1+\lambda_i^2}.\]

If $n>m$, then $\lambda_1\geq 1\geq \lambda_2\geq ....\geq \lambda_m\geq 0=\lambda_{m+1}...=\lambda_{n}$. If $n\leq m$, then we can fill in $0$ as $\lambda_{n+1},...,\lambda_{n+m}$ and simply take $  u^{n+1}\equiv 1,...,   u^{n+m}\equiv 1$, so that $\lambda_1\geq 1\geq \lambda_2\geq ....\geq \lambda_n\geq 0=\lambda_{n+1}...=\lambda_{m}$. In both cases, we deduce for $a=4m\geq 1$ that
\begin{align}\label{eqn:main-ineq}
0\geq & -a\eta \sum_{\alpha=n+1}^{n+m} (  u^\alpha)^2 +  \sum_{i=1}^m \frac{2a^2\eta (  u^{n+i})^2\lambda_i^2}{1+\lambda_i^2}  -\sum_{i=1}^m \frac{ 8a \lambda_i   u^{n+i}}{1+\lambda_i^2} \\
=& \sum_{i=1}^m \left(-4m (  u^{n+i})^2\eta +32m^2  \frac{\lambda_i^2}{1+\lambda_i^2} (  u^{n+i})^2\eta - 32m \frac{\lambda_i }{1+\lambda_i^2}  u^{n+i}  \right) \nonumber\\
=& \sum_{i=1}^m\frac{(32m^2-4m)\lambda_i^2 (  u^{n+i})^2\eta -32m \lambda_i   u^{n+i} - 4m (  u^{n+i})^2\eta }{1+\lambda_i^2} \nonumber\\
\geq & \sum_{i=1}^m\frac{16m^2\lambda_i^2 \eta -64m \lambda_i  - 16m \eta }{1+\lambda_i^2} \nonumber\\
=&16m \sum_{i=1}^m\frac{m\eta\lambda_i^2  -4 \lambda_i  -  \eta }{1+\lambda_i^2}, \nonumber
\end{align}    

where in the latter $\geq$, we have used the fact that $1\leq   u^{n+i}\leq 2$. We note that at $p$, $\eta\in (0,1]$.

Next, we argue that if the largest singular value $\lambda_1(p)$ becomes excessively large, then the right hand side of the aforementioned expression \eqref{eqn:main-ineq} will possess a positive lower bound, which leads to a contradiction. For this purpose, we utilize a simple calculus lemma:
  
\begin{lma}\label{lem:Calc}
For a fixed $\kappa\in(0,1]$ and positive integer $m$, the function $h(s)=\dfrac{m\kappa s^2-4s-\kappa}{1+s^2}$ is non-increasing on the interval $\left[0, s_*  \right]$ and non-decreasing on the interval $\left[ s_* ,\infty  \right)$, where $s_*$ is given by

\[s_*=\dfrac{\sqrt{\kappa^2(m+1)^2+4^2}-\kappa(m+1)}{4}. \]
    
\end{lma}

\begin{proof}
    By direct differentiation, we have
    \begin{align*}
        h'(s)=\frac{4s^2+2(m+1)\kappa s -4}{(1+s^2)^2},
    \end{align*}
    which has a unique positive root $s_*$ and $h'(s)$ is non-positive on $[0,s_*]$  and non-negative on $[s_*,\infty)$.
\end{proof}

Now suppose $\eta\lambda_1 (p)\geq  8m$, so we have the following consequences at point $p$ since $\eta(p)\in (0,1]$ and $\lambda_1\lambda_i< 1$ for $i\geq 2$:
\begin{equation}
    \begin{cases}
\lambda_1 &\geq \dfrac{8m}{\eta} \geq   \dfrac{\sqrt{\eta^2(m+1)^2+4^2}+\eta(m+1)}{4}>s_*    \\
\lambda_i &\leq \dfrac{\eta}{8m}\leq \dfrac{4}{\sqrt{\eta^2(m+1)^2+4^2}+\eta(m+1)}=s_*, \quad \text{for}\quad i\geq 2.
    \end{cases}
\end{equation}

Apply Lemma \ref{lem:Calc}, we see
\begin{equation} \label{eqn:main-ineq-2}
    \begin{cases}
h(\lambda_1)\geq \dfrac{m\eta \left(8m/\eta  \right)^2-4\left(8m/\eta  \right)-\eta}{1+\left(8m/\eta  \right)^2}  =   \dfrac{1}{\eta} \dfrac{64m^3-32m-\eta^2   }{1+\left(8m/\eta  \right)^2}  \\
h(\lambda_i)\geq  \dfrac{m\eta \left(\eta/8m  \right)^2-4\left(\eta/8m  \right)-\eta}{1+\left(\eta/8m  \right)^2} =   \dfrac{1}{\eta} \dfrac{m\eta^2-32m-64m^2   }{1+\left(4m/\eta  \right)^2}  , \quad \text{for}\quad i\geq 2.
    \end{cases}
\end{equation}

However, \eqref{eqn:main-ineq} and \eqref{eqn:main-ineq-2} imply
\begin{align*}
   0\geq  &\sum_{i=1}^m\frac{m\eta\lambda_i^2  -4 \lambda_i  -  \eta }{1+\lambda_i^2}=\sum_{i=1}^m h(\lambda_i)\\
    \geq & \dfrac{1}{\eta} \dfrac{64m^3-32m-\eta^2   }{1+\left(8m/\eta  \right)^2}  +  \dfrac{1}{\eta} \dfrac{m\eta^2-32m-64m^2   }{1+\left(8m/\eta  \right)^2} \times (m-1)\\
    =& \frac{1}{\eta} \dfrac{32m^2+(m^2-m-1)\eta^2  }{1+\left(8m/\eta  \right)^2}  >0,
\end{align*}

which is a contradiction! This means, if area decreasing condition holds at max point $p$ of $\phi w=\eta e^{-4m\frac{|  u|^2}{t}}   w$, then largest eigenvalue $\lambda_1(p)$ satisfies
\[\eta(p)\lambda_1(p)\leq 8m.   \]

Notice that $|d   u|= \sqrt{\sum_i\lambda_i^2} \leq   w^n$  and $  w\leq 2\lambda_1$ , so we get

\[\eta e^{-4m\frac{|  u|^2}{t}}   w  \leq  2\eta (p) \lambda_1 (p)\leq 16m.\]

By the fact that $|  u|^2\leq 4m$, this implies
\[ \left(\frac{1}{(1+2\Lambda)^2}-|\textbf{x}|^2-2nt \right)^n |d  u|(\textbf{x},t)\leq (16m)^n  e^{4nm\frac{|  u|^2}{t} } \leq (16m)^n  e^{16nm^2/t }  . \]

Recall from our earlier discussion that we have made a notational simplification, where the $u$ functions referred to here actually represent the height functions of the transformed graphical mean curvature flow, as explained in Subsection \ref{subsec:parabolic-rescaling}. Building upon this, we have successfully demonstrated that if the transformed map $\bar u(\cdot,\bar t): B^n_{\frac{1}{1+2\Lambda}}\rightarrow [1,2]^m$ satisfies area-decreasing condition at $\bar {\textbf{x}}$ for $\bar t \in \left(0,\frac{1}{(1+2\Lambda)^2}\right]$, then we can derive the following estimate:

\[ \left(\frac{1}{(1+2\Lambda)^2}-|\bar{\textbf{x}}|^2-2n\bar t \right)^n |d  \bar u|(\bar{ \textbf{x}},\bar t)\leq  (16m)^n  e^{16nm^2/\bar t}.   \]

Transform this estimate back to the original $u$ map via \eqref{eqn:height-function-transform} and we conclude that if   $u(\cdot,t):B_1^n\rightarrow \mathbb R^m$ is area decreasing at $\textbf{x}$ and $t\in (0,1]$, then
\[ \left(\frac{1}{(1+2\Lambda)^2}-\frac{|\textbf{x}|^2}{(1+2\Lambda)^2}-\frac{2nt}{(1+2\Lambda)^2} \right)^n |du|(\textbf{x},t)\leq  (16m)^n  e^{16nm^2(1+2\Lambda)^2  / t}   . \]

In particular, by Corollary \ref{cor:area-decreasing}, we know  $u\left(\cdot,\frac{1}{4n} \right)$ is area decreasing at origin $\textbf{0}$. So at the $ \left(\textbf{0}, \frac{1}{4n} \right)\in B_1^n\times (0,1]$, we have
\[|du| \left(\textbf{0}, \frac{1}{4n} \right) \leq  (32m)^n (1+2\Lambda)^n e^{64n^2m^2 (1+2\Lambda)^2 },\]

where $\Lambda= ||u(\cdot,0)||_\infty $. This implies the desired estimate:
\begin{equation}
        |du|(\textbf{0},\frac{1}{4n})\leq K_1 e^{K_2 ||u(\cdot,0)||^2_\infty},
\end{equation}

    for some constants $K_1,K_2>0$ that only depend on $n,m$. This concludes the proof of Theorem \ref{thm:main}.

\appendix

\section{Interior gradient estimate for Minimal Surface System}\label{app:MSS}
In this appendix, we provide a Korevarr type maximum principle proof of the interior gradient estimate for Minimal Surface System assuming the area non-increasing condition, which is self-contained. We remark that the result was already proven in \cite{W2} using two different methods: an integral method and a Korevarr type maximum principle method. The integral proof in \cite{W2} gave a robust and correct demonstration of the sharp estimate while the maximum principle proof is incomplete and we intend to give a refined version here.

We consider $u=(u^{n+1},...,u^{n+m}):B_1\subset \mathbb R^n\longrightarrow  (-\infty,-1]^m$, after a translation $u^\alpha\mapsto u^\alpha-||u^\alpha||_\infty-1$ for each $\alpha=n+1,...,n+m$, solving the minimal surface system 
\begin{equation}\label{eqn:MSS}
    \sum_{i,j=1}^n g^{ij} \frac{\partial^2 u^\alpha}{\partial x^i \partial x^j}=0, \quad \text{for each}\quad \alpha=n+1,...,n+m,
\end{equation}

where $g^{ij}=(g_{ij})^{-1}$ and induced metric is given by
\[g_{ij}=\delta_{ij} + \sum_{\beta=n+1}^{n+m}\frac{\partial u^\beta}{\partial x^i} \frac{\partial u^\beta}{\partial x^j},  \]

and the induced volume element is given by
\[v=\sqrt{\det (g_{ij})}=\sqrt{\det\left(I + (du)^T du \right)}.\]

We denote the graph of of $u$ in $\mathbb R^{n+m}$ over $B_1$ by $\mathfrak S$, and we will always use the following Laplace operator induced by $g$: 
\[\Delta=g^{ij}\frac{\partial^2 }{\partial x^i \partial x^j}, \]
where $x^i$ is Euclidean coordinate system on $B_1\subset \mathbb R^n.$

Our goal is to establish an upper bound for $|du^\alpha|$ in terms of $|u^\alpha|$ at an interior point for $u^\alpha$ solving \eqref{eqn:MSS}, assuming area non-increasing condition of the map $u$.

We prefer using a tilted function $\Tilde{F}=\Tilde{F}(x,y)$ to denote a function defined on $\mathbb R^n\times\mathbb R^m$, where $\{x^i, i=1,...,n\}$ is a coordinate system on $\mathbb R^n$ and $\{y^\alpha,\alpha=n+1,...,n+m\}$ is a coordinate system on $\mathbb R^m$. $\Tilde{F}_i$ and $\Tilde{F}_\alpha$ represent regular partial derivatives $\frac{\partial \Tilde{F}}{\partial x^i}$ and $\frac{\partial \Tilde{F}}{\partial y^\alpha}$, same holds for higher order partial derivatives.

The main theorem of this appendix can be summarized as follows
\begin{thm}[\cite{W2}]\label{thm:MSS-gradient-estimate}
    Let $u=(u^{n+1},...,u^{n+m}):B_1\subset \mathbb R^n\longrightarrow  (-\infty,-1]^m$ be a $C^2$ solution to minimal surface system \eqref{eqn:MSS} s.t. the Jacobian of $du:\mathbb R^n\rightarrow \mathbb R^m$ on any two dimensional subspace of $\mathbb R^n$ is less than or equal to one, then we have at the origin, we have estimate
    \begin{equation}
        |du|(0)\leq K_1 e^{K_2 |u(0)|^2},
    \end{equation}
    where $K_1,K_2$ only depends on $n$.
\end{thm}

\begin{rem}
    By translation and re-scaling, we should be able to get the following estimate for any domain $\Omega\subset \mathbb R^n$ and any $x_0\in \Omega$:
    \begin{equation}
        |du|(x_0)\leq K_1 e^{K_2 |u(x_0)|^2/d},
    \end{equation}
    where $d=\text{dist}(x_0,\partial \Omega)$.
\end{rem}

\begin{rem}
    Constants $K_1,K_2$ can be chosen to only depend on $n$, because the rank of $du$ is always bounded above by $n$.
\end{rem}

\begin{rem}
    Similar to the hypersurface case, such estimate is no sharp in the sense that on can prove 
    \[ |du|(0)\leq K_1 e^{K_2 |u(0)|},\]
    under the same assumption using integral method. For further details, we refer readers to Section 3 of \cite{W2}.
\end{rem}

\subsection{Key assumption: area non-increasing}

Key condition we are imposing is that the Jacobian of $du:\mathbb R^n\rightarrow \mathbb R^m$ on any two dimensional subspace of $\mathbb R^n$ is less than or equal to one. ($u$ is an area non-increasing map.)

This assumption can be described in terms of singular values of $du$, i.e. eigenvalues $\lambda_i\geq 0$ of $\sqrt{(du)^Tdu}$ at any fixed point $x_0\in B_1$. The assumption is equivalent to say, any any point, we have for $1\leq i\neq j\leq n$,
\[\lambda_i\lambda_j \leq 1.\]

We can always give an ordering to the singular values that $\lambda_1\geq \lambda_2\geq...\geq \lambda_n\geq 0$.

The area non-increasing condition implies $\lambda_1\geq 1 \geq \lambda_2\geq...\geq \lambda_n\geq 0$. So under such condition, the gradient estimate reduces to bound $\lambda_1$.

\subsection{Key Lemma:} Under the assumption of area non-increasing, $w=v^{\frac{1}{n}}$ satisfies an important partial differential inequality.

\begin{lma}[cf. Lemma 2.2 in \cite{W2}]If $\lambda_i\lambda_j\leq 1$ for $i\neq j$, then induced volume element $v$ satisfies
    \[\Delta \log v \geq \frac{1}{n}|\nabla \log v |^2,  \]
    where $\nabla$ and $\Delta$ denote the gradient and Laplacian of the induced metric $g_{ij}$. Equivalently, we have $w=v^{\frac{1}{n}}$ satisfies
    \begin{equation}\label{eqn:w-subharmonic}
        \Delta w \geq \frac{2|\nabla w|^2}{w}.
    \end{equation}
\end{lma}

We remark that, $w$ controls the gradient $|du|$ from above pointwisely: using singular values,
\[|du|=\sqrt{\sum_{i=1}^n \lambda_i^2}\leq \sqrt{\prod_{i=1}^n (1+\lambda_i^2)}=v=w^n. \]

\subsection{Maximum Principle}

We assume that $u=(u^{n+1},...,u^{n+m}):B_1\subset \mathbb R^n\longrightarrow  (-\infty,-1]^m$ solves \eqref{eqn:MSS}. Let $\Tilde{\eta}=\Tilde{\eta}(x,y)$ be a cut-off function that is non-negative and continuous on $B_1\times (-\infty,-1]^m$ and it's zero on $\{(x,y)\in \mathbb R^n\times \mathbb R^m: |x|=1,\ y^\alpha<-1 \}$.

Let $\eta(x)=\Tilde{\eta}(x,u(x))$ be the restriction of $\Tilde{\eta}$ to the graph $\mathfrak S$ of $u$. Now the function $\eta w$ achieves a positive maximum in the interior point of $B_1$, say $p$. Then at this point $p$, we have
\begin{equation*}
    \begin{cases}
        \nabla(\eta w)=0,\\
        \Delta(\eta w)\leq 0.
    \end{cases}
\end{equation*}

The first equation implies at $p$,
\[\nabla \eta=-\frac{\eta}{w}\nabla w,\]

and the second equation implies at $p$,
\[w\Delta \eta +2\nabla \eta\cdot \nabla w + \eta \Delta w\leq 0. \]

Combining these two equations and \eqref{eqn:w-subharmonic} under the assumption of area non-increasing map, we get at $p$,
\begin{equation}\label{eqn:laplace-eta}
    \Delta \eta \leq 0.
\end{equation}

\subsection{Specifying cut-off}

Following Korevarr's ansatz in \cite{K}, we write $\Tilde{\eta}=f\circ \Tilde{\phi}$, where $f$ is a smooth single variable, increasing, strictly convex function with $f(0)=0$, and $\Tilde{\phi}=\Tilde{\phi}(x,y)$ is a real valued function defined on $B_1\times (-\infty,-1]^m$. Hence $\eta(x)=f\circ \Tilde{\phi}(x,u(x))$ and we can compute easily for $1\leq i,j\leq n$:
\begin{equation}\label{eqn:derivatives-eta}
    \begin{split}
\eta_i=&f'\cdot (\Tilde{\phi}_i+\Tilde{\phi}_\alpha u^\alpha_i), \\   
\eta_{ij}=&f''\cdot (\Tilde{\phi}_i+\Tilde{\phi}_\alpha u^\alpha_i) (\Tilde{\phi}_j+\Tilde{\phi}_\beta u^\beta_j) \\
&+f'\cdot (\Tilde{\phi}_{ij}+\Tilde{\phi}_{i\alpha} u^\alpha_j+\Tilde{\phi}_{j\alpha} u^\alpha_i+\Tilde{\phi}_{\alpha\beta} u^\alpha_iu^\beta_j+\Tilde{\phi}_\alpha u^\alpha_{ij}),
    \end{split}
\end{equation}
where $\eta_i,\eta_{ij}$, $u_i^\alpha$, $u_{ij}^\alpha$ are evaluated at $x$; $f',f''$ are evaluated at $\Tilde{\phi}(x,u(x))$ and $\Tilde{\phi}_i,\Tilde{\phi}_\alpha,\Tilde{\phi}_{ij},\Tilde{\phi}_{i\alpha},\Tilde{\phi}_{\alpha\beta}$ are evaluated $(x,u(x))$.

Therefore, at the maximum point $p$ of $\eta w$, using the \eqref{eqn:MSS}, \eqref{eqn:laplace-eta} and \eqref{eqn:derivatives-eta}, we have
\begin{align}\label{eqn:laplace-eta-tilephi}
   0\geq & \Delta \eta = g^{ij}\eta_{ij}\\ 
     =&f''\cdot g^{ij}(\Tilde{\phi}_i+\Tilde{\phi}_\alpha u^\alpha_i) (\Tilde{\phi}_j+\Tilde{\phi}_\beta u^\beta_j) \nonumber\\
     &+f'\cdot g^{ij}(\Tilde{\phi}_{ij}+\Tilde{\phi}_{i\alpha} u^\alpha_j+\Tilde{\phi}_{j\alpha} u^\alpha_i+\Tilde{\phi}_{\alpha\beta} u^\alpha_iu^\beta_j+\Tilde{\phi}_\alpha \cancel{u^\alpha_{ij}}). \nonumber
\end{align}

Up to here, we haven't specified what $\tilde \phi$ is. Before doing that, we can make use of diagonalization at $p$ to simplify the expression. We can choose coordinate systems $x^i$ and $y^\alpha$ appropriately such that
\[g^{ij}=\frac{1}{1+\lambda_i^2}\delta_{ij}, \quad u^\alpha_i=\lambda_i \delta_{\alpha,n+i}. \]

Because $f',f''>0$ by assumption, we see from \eqref{eqn:laplace-eta-tilephi},
\begin{align*}
   0\geq & f''\sum_{i=1}^n\frac{\delta_{ij}}{1+\lambda_i^2} \left(\tilde \phi_i + \tilde\phi_{n+i} \lambda_i   \right)\left(\tilde \phi_j + \tilde\phi_{n+j} \lambda_j   \right)+f' g^{ij}(\Tilde{\phi}_{ij}+\Tilde{\phi}_{i\alpha} u^\alpha_j+\Tilde{\phi}_{j\alpha} u^\alpha_i+\Tilde{\phi}_{\alpha\beta} u^\alpha_iu^\beta_j).
\end{align*}

Right now, we specify our choice of $\tilde\phi$: for a fixed number $u_0\geq 1$, we set
\begin{equation}\label{eqn:cutoff-phi}
    \tilde\phi(x,y)=\left(\frac{1}{2u_0}\sum_\alpha y^\alpha +1 -|x|^2 \right)^+.
\end{equation}

At the maximum point $p$ of $\eta w$, $\eta(p)>0$ so $\tilde\phi(p,u(p))>0$, hence $\tilde\phi$ is smooth in a neighborhood $p$ and we can compute at the point $p$: for $i=1,...,n$
\[\tilde\phi_i=-2x^i,\quad  \tilde\phi_{n+i} =\frac{1}{2u_0},\quad \tilde\phi_{ij}=-2\delta_{ij},\quad \tilde\phi_{i,n+i}=0=\tilde\phi_{n+i,n+i}.  \]

As a result, for this choice of $\tilde\phi$, we see at $p$:
\begin{align}\label{eqn:Cf''+Cf'}
   0\geq & f''\sum_{i=1}^n\frac{1}{1+\lambda_i^2}
   \left(-2x^i+\frac{1}{2u_0}\lambda_i   \right)^2+f' g^{ij}(-2\delta_{ij}) \\
   =& f''\sum_{i=1}^n \frac{(\lambda_i-4u_0x^i )^2 }{4u_0^2(1+\lambda_i^2)} -f'\sum_{i=1}^n\frac{2}{1+\lambda_i^2} \nonumber\\
   \geq &f'' \frac{(\lambda_1-4u_0x^1 )^2 }{4u_0^2(1+\lambda_1^2)}  - 2n f' \nonumber\\
   \geq & f''  \frac{(\lambda_1-4u_0|x^1| )^2 }{4u_0^2(1+\lambda_1^2)}  - 2n f' \nonumber
\end{align}
where in the second inequality, we have used the fact that $f',f''>0$ and we kept the term involving largest singular $\lambda_1$ in the first summation; in the last inequality we have used $(a-b)^2\geq (|a|-|b|)^2$ and $\lambda_1\geq 0$.

Next, we want to argue that if the largest singular value $\lambda_1(p)$ is too big, then the coefficient of $f''$ in the right hand side \eqref{eqn:Cf''+Cf'} will have a positive lower bound, which helps us to derive a contradiction. For this purpose, we need a calculus lemma:
\begin{lma}\label{lem:calc}
   If $c\geq 0$, then the function $h(s)=\dfrac{(s-c)^2}{1+s^2}$ is non-decreasing on the interval $\left[c ,\infty  \right)$. 
\end{lma}

\begin{proof}
This is by direct differentiation. Since $c\geq 0$, we have
    \begin{align*}
        h'(s)=\frac{2(s-c)(1+cs)}{(1+s^2)^2}\geq 0\quad \text{on}\quad [c,\infty).
    \end{align*}

\end{proof}

Now suppose $\lambda_1\geq 8u_0>4u_0 |x^1|$, so we have the following consequences by Lemma \ref{lem:calc} for $c=4u_0|x^1|$: since $|x^1|\leq 1$,
\begin{equation}
 \frac{(\lambda_1-4u_0|x^1| )^2 }{1+\lambda_1^2}\geq  \frac{(8u_0-4u_0|x^1|)^2}{1+(8u_0)^2}\geq \frac{16 u_0^2}{1+ (8u_0)^2}
\end{equation}

Therefore, by \eqref{eqn:Cf''+Cf'}, $\lambda_1(p)\geq 8u_0$ implies at $p$,
\begin{align*}
    0
   \geq & f''  \frac{(\lambda_1-4u_0|x^1| )^2 }{4u_0^2(1+\lambda_1^2)}  - 2n f'\\
   \geq & f'' \frac{4}{1+(8u_0)^2}-2n f'\\
   \geq &\frac{1}{100u_0^2}f'' - 2n f',\quad \text{as}\quad u_0\geq 1.
\end{align*}

If we have taken $f(t)=e^{C_1t}-1 $, with $C_1=300n u_0^2$, then we have that $\lambda_1(p)\geq 8u_0$ implies
\[0\geq \frac{1}{100u_0^2}C_1^2-2n C_1=300n^2u_0^2>0,\]

which is a contradiction!

Now let's summarize what we have proven: at the maximum point $p$ of $\eta(x)w(x)$ where $\eta(x)=e^{C_1\tilde\phi(x,u(x))}-1$, $\tilde\phi$ is given by \eqref{eqn:cutoff-phi}, we have upper bound for largest singular value $\lambda_1(p)\leq 8u_0$. Hence 
\[w(p)=\left(\prod_{i=1}^n (1+\lambda_i(p)^2)\right)^{\frac{1}{2n}}\leq 2^{\frac{n-1}{2n}} (1+64u_0^2)^{\frac{1}{2n}}=:C_2.\]

Therefore, for any $x\in B_1$, we have
\[\eta(x)w(x)\leq \eta(p)w(p)\leq C_2e^{C_1}.\]

To get estimate at the origin $x=0$, we take $u_0=-\sum_\alpha u^\alpha(0)\geq 1$ (recall we have done translation s.t. each $u^\alpha\leq -1$), then 
\begin{align*}
  &\tilde\phi(0,u(0))= \left(\frac{1}{2u_0}\sum_\alpha u^\alpha(0) +1 -0 \right)^+=\frac{1}{2} \\
  \Longrightarrow &\eta(0)=e^{\frac{1}{2}C_1}-1,\\
  \Longrightarrow & \left(e^{\frac{1}{2}C_1}-1\right) w(0)\leq C_2 e^{C_1}, 
\end{align*}

where $C_1$ depends quadratic-ly on $u_0$ and $C_2$ depends polynomial-ly on $u_0$. This immediately implies
\[|du|(0)\leq K_1 e^{K_2' u_0^2}\leq K_1 e^{K_2 |u(0)|^2},\]
where $K_1,K_2$ only depends on $n$. This finishes the proof of Theorem \ref{thm:MSS-gradient-estimate}.


\begin{thebibliography}{99}

\bibitem[BDM]{BDM} E. Bombieri, E. De Giorgi, and M. Miranda, Una maggiorazione a priori relativa alla ipersuperfici minimali non parametriche, \textit{Arch. Rational Mech. Anal.}, 32 (1969) 255--267.

\bibitem[CM]{CM} T.H. Colding and W.P. Minicozzi II, Sharp estimates for mean curvature flow of graphs, \textit{J. Reine Angew. Math.}, 574 (2004) 187-195.



\bibitem[EH]{EH} K. Ecker and G. Huisken, Interior estimates for hypersurfaces moving by mean curvature, \textit{Invent. Math.}, 105 (1991) 547--569.

\bibitem[F1]{F1} R. Finn, On equations of minimal surface type, \textit{Annals of Math.}, 60 (1954) 397--416.

\bibitem[F2]{F2} R. Finn, Remarks relevant to minimal surfaces, and to surfaces of prescribed mean curvature, \textit{J. Analyse Math.}, 14 (1965) 265--296.

\bibitem[GiTr]{GiTr} D. Gilbarg and N. Trudinger, Elliptic partial differential equations of second order, Springer Verlag (1983).

\bibitem[K]{K} N. Korevaar, An easy proof of the interior gradient bound for solutions to the prescribed mean curvature equation, Nonlinear functional analysis and its applications, Part 2, \textit{Proc. of Symposia in Pure Math.} 45 (1986) 81--89.

\bibitem[LTW]{LTW} M.-C. Lee, L.-F. Tam, and J. Wan, Rigidity of area non-increasing maps, \textit{arXiv preprint arXiv:2312.10940} (2023).


\bibitem[TW]{TW} M.-P. Tsui and M.-T. Wang, Mean curvature flows and isotopy of maps between spheres, \textit{Commun. Pure Appl. Math.}, 57(8) (2004) 1110-1126.

\bibitem[TTW]{TTW} C.-J. Tsai, M.-P. Tsui, and M.-T. Wang, A new monotone quantity in mean curvature flow implying sharp homotopic criteria, arXiv preprint arXiv:2301.09222 (2023).

\bibitem[W1]{W1} M.-T. Wang, Long-time existence and convergence of graphic mean curvature flow in arbitrary codimension, \textit{Inventiones mathematicae} 148.3 (2002) 525-543.



\bibitem[W2]{W2} M.-T. Wang, Interior gradient bounds for solutions to the minimal surface system, \textit{American Journal of Mathematics} 126.4 (2004) 921-934.



\end{thebibliography}
\end{document}